\documentclass[12pt]{amsart}
\usepackage{amsmath}
\usepackage{amsthm}
\usepackage{amssymb}
\usepackage{verbatim}
\usepackage{subfigure}
\usepackage{verbatim}
\usepackage{url}
\numberwithin{equation}{section}
\usepackage{textcomp}
\usepackage{stmaryrd}
\usepackage{xy}
\usepackage{color}
\usepackage[pdftex]{graphicx}
\usepackage{mathrsfs}
\definecolor{Mygrey}{gray}{0.75}

\def\displayandname#1{\rlap{$\displaystyle\csname #1\endcsname$}%
                      \qquad \texttt{\char92 #1}}

\makeatletter
\def\url@leostyle{%
  \@ifundefined{selectfont}{\def\UrlFont{\sf}}{\def\UrlFont{\small\ttfamily}}}
\makeatother
\newtheorem{thm}{Theorem}[section]
\newtheorem{pro}[thm]{Proposition}
\newtheorem{lem}[thm]{Lemma}
\newtheorem{cla}[thm]{Claim}

\newtheorem{cor}[thm]{Corollary}

\theoremstyle{definition}
\newtheorem{df}[thm]{Definition}

\theoremstyle{remark}
\newtheorem{rem}[thm]{Remark}
\newtheorem{exa}[thm]{Example}

\begin{document}

\bibliographystyle{acm}

\title{FACTORIZATION OF TROPICAL  MATRICES}
  \author[Adi Niv]{ Adi Niv$^\dagger$ }

  \thanks{$^\dagger$ Department of Mathematics, Bar-Ilan University, Ramat Gan 52900, Israel.\newline
Email:  {\tt adi.niv@live.biu.ac.il}} 
\thanks{This paper is part of the author's PhD thesis, which was written at Bar-Ilan University under the supervision of Prof. L.\ H.\ Rowen.} 
\thanks{We thank the referee for helpful comments on the original version of this paper.}

\begin{abstract} In contrast to the situation in classical linear algebra,  not every tropically non-singular matrix can be factored into a product of tropical elementary matrices. We  do prove the factorizability of any tropically non-singular $2\times 2$ matrix and, relating to the existing Bruhat decomposition, determine which  $3\times 3$ matrices are factorizable. Nevertheless, there is a closure operation, obtained by means of the tropical adjoint, which is always factorizable, generalizing the decomposition of the closure operation~$*$ of a matrix.
\end{abstract}

\maketitle

%First page headline in LaTeX for S\'eminaire Lotharingien de Combinatoire
%--first part
\thispagestyle{myheadings}
\font\rms=cmr8
\font\its=cmti8
\font\bfs=cmbx8

\markright{}
\def\thepage{}

 \section{Introduction}

The tropical semifield is an ordered group $\mathcal{G}$ (usually the set of real numbers $\mathbb{R}$ or the set of rational numbers $\mathbb{Q}$),  together with $-\infty$, denoted as $\mathbb{T}=\mathcal{G}\bigcup\{-\infty\}$, and equipped with the operations $a\varoplus b=max\{a,b\}$ and $a\varotimes b=a+b$, denoted as $a+b$ and~$ab$ respectively (see ~\cite{MPA}, ~\cite{TAG}  and ~\cite{TIM}). This arithmetic enables one to simplify non-linear questions by answering them in a linear setting (see~\cite{PA}), which applied in discrete mathematics, optimization, algebraic geometry and more, as has been well reviewed in ~\cite{NWR}, ~\cite{LOP}, ~\cite{MAA}, ~\cite{S&A}, ~\cite{MPW}, ~\cite{IM} and ~\cite{TGA}. 

This structure can also be studied via the valuation over the field $K=\mathbb{C}\{\{t\}\}$ of Puiseux series to the ordered group $(\mathbb{Q},+,\geq )$, as has been done in ~\cite{TLI}, by looking at the dual structure~$trop(a)=-val(a)$ denoted as the tropicalization of $a\in K$. In order to make the connection between the results in the work of Buchholz in ~\cite{TLI} and the results in this paper we say~$trop(a+b)=max\{trop(a),trop(b)\}$. Then it is obvious that the tropical structure deals with the uncertainty of $a=b$ in the valuation, in the form of~$trop(a+a)=trop(a)$.

\vskip 0.25 truecm

In this paper we aspire to solve the tropical factorization problem raised in ~\cite{TLI}and ~\cite{CTF}, by passing to a wider structure called the supertropical semiring (see ~\cite{STA} and ~\cite{STLA}), equipped with the ghost ideal~$G=\mathcal{G}^{\nu}$. We denote  as~$R=T\bigcup G\bigcup \{-\infty \}$ the supertropical semiring, where $T=\mathcal{G}$, which contains the so called tangible elements of the structure and~$\forall a\in T, \ a^{\nu}\in G$ are the ghost elements of the structure, as defined in ~\cite{STA}. So $G$ inherits the order of $\mathcal{G}$. We distinguish between a maximal element $a$ that is being attained once, i.e.~$a\in T$ which is invertible, and a maximum that is being attained at least twice, i.e.~$a+a=a^{\nu}\in G$, which  is not invertible. 

The work in ~\cite{STMA}, ~\cite{STMA2} and ~\cite{STMA3} shows that even though the semiring of matrices over the supertropical semiring lacks negation, it satisfies many of the classical matrix theory properties when using the ghost ideal $G$. We say $a$ ghost surpasses $b$, denoted by $a\vDash_{gs} b$, if $a=b$ or~$a\in G$ and $a^{\nu}>b^{\nu}$. We say $a$ is $\nu$-equivalent to~$b$, denoted by~$a\cong_{\nu}b$, if~$a^{\nu}=b^{\nu}$. That is, in the tropical structure, $\nu$-equivalent means equal.   
\vskip 0.25 truecm

\begin{df}
We define a matrix~$A\in M_n(\mathbb{T})$ to be  \textbf{tropically singular} if there exist at least two different permutations that attain the maximum value in the determinant. Otherwise the matrix is \textbf{tropically non-singular}.

Consequently  a matrix~$A\in M_n(R)$ is singular if~$det(A) \in G\bigcup\{-\infty\}$and non-singular if~$det(A) \in T$. A matrix $A$ is \textbf{strictly singular} if $det(A)=-\infty$.  
\vskip 0.25 truecm

Notice that over the tropical semifield we cannot indicate if the matrix is tropically non-singular by the value of its determinant, which is always invertible over $\mathbb{T}\setminus\{-\infty\}$. Over the supertropical semiring however, a supertropically non-singular matrix will have an invertible determinant, while a supertropically singular matrix will have a non-invertible determinant. 

As the singularity definitions are identical over the tropical and supertropical structures, we will only indicate "non-singular" or "singular" and "over the structure $\mathbb{T}$" or "over the structure $R$" (which will effect the value and invertability of the determinant). When talking about non-singular matrices over the field $K$ we use "invertible".
\end{df}

Looking at the theorem of  tropical determinants (defined in ~\cite{TA} to be the usual permanent) over the supertropical semiring, which satisfies $det(AB)\vDash_{gs} det(A)det(B)$ (~\cite[Theorem 3.5]{STMA}), one might wonder why the product of two  non-singular matrices maybe  singular. In this work we attempt to understand the reason by investigating the elementary matrices as the \textgravedbl generators\textacutedbl of matrices, in analogy to the well known classical fact that over a field $K$, $GL_n(K)$ is generated by elementary matrices. 

This situation is subtler for matrix semirings over a tropical semiring, as shown in~\S 3 and ~\cite[Lemma 4.36]{TLI}. Whereas every~$2\times 2$ non-singular matrix is factorizable, this fails for~$3\times 3$ matrices. However, we salvage a positive result, described in Corollary~6.6,  by passing to~$adj(A)$. We show how this applies to the closure operation~$*$ in ~\cite{LDM}, by establishing that the power~$A^k$ of a matrix~$A\in SL_n(\mathbb{T})$, with~$1_{\mathbb{T}}$ on its diagonal, stabilizes at~$k=n-1$.

%First page headline in AmS-LaTeX for S\'eminaire Lotharingien de Combinatoire
%--restoring the headers and pagenumbering
\pagenumbering{arabic}
\addtocounter{page}{1}
\markboth{\SMALL ADI NIV }{\SMALL FACTORIZATION OF TROPICAL MATRICES}

\section{preliminaries}

\noindent In this section we establish some fundamental definitions for our work as well as give a glance for the Bruhat decomposition found in ~\cite{TLI}.

\begin{df}
Let $\mathbb{T}^n$ be the free module of rank $n$ over the tropical  semifield, and $R^n$ be the free module of rank $n$ over the supertropical  semiring. We define the \textbf{standard base} of $\mathbb{T}^n$, and therefore of $R^n$, to be $e_1,...,e_n$ where 

$e_i=
\begin{cases}
1_{\mathbb{T}}=1_R,\ \text{in the}\ i^{th}\ \text{coordinate}\\
0_{\mathbb{T}}=0_R,\ \text{otherwise} 
\end{cases}$.
\end{df}

\begin{df} The tropical \textbf{identity matrix} in the tropical matrix semiring is 

\noindent the~$n\times n$ matrix with the standard base for its columns. We denote this matrix as $$I_{\mathbb{T}}=I_R=I.$$
\end{df}

\begin{df}
A matrix $A\in M_n(R)$ is \textbf{tropically invertible} if there exists a matrix~$B\in M_n(R)$ such that $$AB=BA=I.$$
\end{df}

\begin{df}
Corresponding to the three elementary row matrix operations, we  define respectively three types of tropical \textbf{elementary matrices} obtained by applying one such operation to the identity matrix.
We denote these matrices as follows:
\vskip 0.25 truecm
$$E_{i,j}=(a_{t,l}),\ \text{where }a_{t,l}=
\begin{cases}
1_R,\ \text{where}\ t=l\ne i,j\\ 
1_R,\ \text{where}\ i=t\ne l=j\ or\ j=t\ne l=i\\
0_R,\ \text{otherwise}
\end{cases}$$
\vskip 0.25 truecm
\noindent which means switching the $i^{th}$ and $j^{th}$ rows.

\vskip 0.5 truecm

$$E_{k\cdot (i^{th} row)}=(b_{t,l}),\ \text{where }b_{t,l}=
\begin{cases}
1_R,\ \text{where}\ t=l\ne i\\
k,\ \text{where}\ t=l=i\\
0_R, \text{where}\ t\ne l
\end{cases}$$
\vskip 0.25 truecm
\noindent which means multiplying the $i^{th}$ row by an invertible $k \in T$.

\vskip 0.5 truecm

$$E_{i+k\cdot(j^{th}row)}=(c_{t,l}),\ \text{where }c_{t,l}=
\begin{cases}
1_R,\ \text{where}\ t=l\\
k,\ \text{where}\ i=t\ne l=j\\
0_R,\ \text{otherwise}
\end{cases}$$
\vskip 0.25 truecm
\noindent which means adding the $j^{th}$ row, multiplied by $k$, to the $i^{th}$ row, where $k \in T$. (We can define these matrices for $k\in R\setminus \{0_R\}$, but since applying $E_{i+k\cdot(j^{th}row)}$ for some $k\in G$ would be the same as applying $E_{i+a\cdot(j^{th}row)}$ twice for some $a\in T$ such that~$a^\nu =k$, we can reduce our set of elementary matrices to these definitions).
\vskip 0.25 truecm
\noindent We refer to the matrices $E_{i,j}$ as elementary matrices of type 1, to the matrices~$E_{k\cdot(i^{th}row)}$ as elementary matrices~of type 2, and to the matrices $E_{i+k\cdot(j^{th}row)}$ as elementary matrices~of
type 3.
\end{df}

\begin{df} A tropically \textbf{factorizable matrix} is defined to be a matrix that can be written as a product of tropical elementary matrices.
\end{df}

%\vskip 0.25 truecm

\begin{df} A square matrix $P_{\pi}=(a_{i,j})$ is defined to be a \textbf{permutation matrix} if there exists $\pi\in S_n$ such that 
$a_{i,j}
\begin{cases}
=0_R, j\ne \pi(i)\\
= 1_R, j=\pi(i)
\end{cases}.$
\end{df} That is, a permutation matrix is a product of elementary matrices of type 1.

By observing the matrices of type 2 as the generators of the diagonal matrices we give the following remark.

\begin{rem}

$ $

\noindent a. A tropical matrix $A$ is invertible if and only if it is a product of elementary matrices of type 1 and 2. That is, a product of a permutation matrix $P_{\pi}$ and a diagonal matrix~$D$, denoted as $D_{\pi}$.
\vskip 0.25 truecm
\noindent b. Non-singular triangular matrices over $R$ and  not strictly singular triangular matrices over~$\mathbb{T}$ are  factorizable.
\end{rem}

\begin{proof} 

$ $

\noindent a. See ~\cite[Proposition 3.9]{STMA}. We unite the products of matrices of type 1 and 2 under the definition of invertible matrices.
\vskip 0.25 truecm

\noindent b. First we can normalize the diagonal to $1_R$, using elementary matrices of type 2. Then, an upper triangular matrix will be obtained by applying 

$$E_{i+a_{i,j}row-j}\ \forall j>i=1,...,n-1 \text{ in this order},$$

\noindent creating one row after another. A lower triangular matrix will be obtained analogously by applying the same elementary operations for $ j<i=2,...,n \text{ in opposite order}.$

\end{proof}

Calculating the determinants of the elementary matrices, one can easily conclude that the product of elementary matrices might yield a singular matrix only when there is an elementary matrix~of type 3 involved in the product. This means that inequality in the rule of determinants arises from elementary matrices~of type~3. However, is it  possible to generate any matrix as a product of elementary matrices? This question is strictly related to the question raised by Buchholz in ~\cite{TLI}: $$\text{Does }trop(AB)=trop(A)trop(B)\text{, where }A,B\text{ are square matrices over the field }K?$$ Meaning, by considering that over a field we are able to factor an invertible matrix, and that the tropicalization of triangular matrices are tropical  triangular matrices,  does the factorization of the tropicalization of a matrix coincides with the tropicalization of the matrix factorization?

In his work, Buchholz states sufficient conditions for a positive answer, by means of  the lowest power of $t$ in $\mathbb{C}\{\{t\}\}$, which will be presented next. In \S 4 we will establish terms for factorizability of $3\times 3$ matrices and show how they relate to Buchholz's conditions. In \S 5 and \S 6 we salvage a positive answer for non-singular matrices over $R$, introducing a closure operation in the supertropical structure. The algorithm, however, also applies to the tropical structure. Moreover, it holds for singular matrices over $\mathbb{T}$, with determinant different than $-\infty$,  as well.

\begin{df}
A \textbf{track of a permutation $\pi\in S_n$} is the sequence $a_{1,\pi(1)}a_{2,\pi(2)}\cdots a_{n,\pi(n)}$ of $n$ entries of the matrix $A=(a_{i,j})\in M_n(R)$.
\end{df}

Let us begin with a motivating example, establishing that a non-singular tropical~$2\times 2$ matrix is always factorizable, determining which  tropical $2\times 2$  matrix has a Bruhat decomposition induced by the decomposition over $K$ and when do the two decompositions coincides.

\begin{exa}

$ $

\noindent a. Let $A$ be a $2\times 2$ invertible matrix over the field $K$. We denote the Bruhat factorization of~$A$ by $PLU$, where $P$ is a product of a diagonal matrix and a permutation matrix, $L$ is a lower unitriangular matrix and~$U$ is an upper unitriangular matrix. 
Then:
\vskip 0.25 truecm

(i) $trop(A)$ is  not strictly singular,
\vskip 0.25 truecm

(ii) $trop(A)=trop(P)trop(L)trop(U)$, when $trop(A)$ is non-singular over $R$,
\vskip 0.25 truecm

(iii) $trop(A)=trop(P)trop(L)trop(U)$, when $trop(A)$ is not strictly singular over $\mathbb{T}$,  

if and only if $trop(det(A))=det(trop(A))$. 
\vskip 0.25 truecm

\noindent (Notice that the determinant on the right hand side is defined to be the permanent)

\vskip 0.25 truecm
\noindent b. Let $B$ be a $2\times 2$ non-singular matrix over $R$, or a $2\times 2$ not strictly singular matrix over~$\mathbb{T}$. Then $B$ is factorizable.

\vskip 0.25 truecm
\noindent c. Let $B$ be a $2\times 2$ non-singular matrix over R and let $A$ be a $2\times 2$ matrix over $K$ such that $trop(A)=B$. Then $A$ is  invertible  and the factorization of $A$ is $PLU$, where~$trop(P)trop(L)trop(U)$  is the factorization of $B$.

\noindent (Meaning, over $\mathbb{T}$, the factorization exists, but might be different than the one being induced by the classical factorization).
\begin{proof}

$ $

\noindent a. $A$ is  invertible. Therefore $|A|\ne 0_K$ and $det(trop(A))\ne 0_R$. If $trop(A)$ is non-singular over~$R$ then one permutation track in $trop(A)$ is strictly bigger than the other. 

In the general  not strictly singular case, we write
$$A=
\left(
\begin{array}{cc}
a_{1,1} & a_{1,2}\\
a_{2,1}  & a_{2,2} 
\end{array}
\right)=P
\left(
\begin{array}{cc}
1_K & a\\
b  &1_K 
\end{array}
\right),$$ where $P$ will relocate and normalize  the tropicalization-source of the dominant permutation track to the diagonal, i.e., $trop(1_K)\geq trop(ab)$. We denote 
$$\bar{A}=\left(
\begin{array}{cc}
1_K & a\\
b  &1_K 
\end{array}
\right).$$

Next, we factor $A$ into $P,L$ and $U$ as follows:
$$A=P
\left(
\begin{array}{cc}
1_K & a\\
b  &1_K 
\end{array}
\right)=P
\left(
\begin{array}{cc}
1_K &0_K\\
b  &1_K-ba 
\end{array}
\right)\left(
\begin{array}{cc}
1_K & a\\
0_K  &1_K 
\end{array}
\right)$$

\noindent and 
$$trop(P)=trop\left(
\begin{cases}
a_{i,j}, j=\pi(i)\\
0_K, otherwise
\end{cases}\right)=
\left(\begin{cases}
trop(a_{i,j}), j=\pi(i)\\
trop(0_K)=-\infty, otherwise
\end{cases}\right)=D_{\pi}$$ (a tropical invertible matrix)
will yield 
\begin{equation}trop(A)=trop(P)trop\left(\left(
\begin{array}{cc}
1_K & a\\
b  &1_K 
\end{array}
\right)\right)=D_{\pi}
\left(
\begin{array}{cc}
0 & trop(a)\\
trop(b)  & 0
\end{array}
\right).\end{equation} Then, since 

1. $trop(det(A))=det(trop(A))$, which is required in case $trop(1_K)=trop(ab)$.

2. $det(A)=det(P)det(\bar{A})$ and 

3. $det(trop(A))=det(D_{\pi})det(trop(\bar{A})$, 

\noindent we get
\begin{equation}D_{\pi}
\left(
\begin{array}{cc}
0 & trop(a)\\
trop(b)  & 0
\end{array}
\right)=D_{\pi}
\left(
\begin{array}{cc}
0 & -\infty\\
trop(b)  & 0
\end{array}
\right)\left(
\begin{array}{cc}
0 & trop(a)\\
-\infty  & 0
\end{array}
\right)\end{equation}\vskip 0.10 truecm
 $$=
D_{\pi}
\left(
\begin{array}{cc}
0 & -\infty\\
trop(b)  & 0+ab
\end{array}
\right)\left(
\begin{array}{cc}
0 & trop(a)\\
-\infty  & 0
\end{array}
\right)=
D_{\pi}
\left(
\begin{array}{cc}
0 & -\infty\\
trop(b)  & det(trop(\bar{A}))
\end{array}
\right)\left(
\begin{array}{cc}
0 & trop(a)\\
-\infty  & 0
\end{array}
\right)$$\vskip 0.10 truecm
 $$
=D_{\pi}
\left(
\begin{array}{cc}
0 & -\infty\\
trop(b)  & trop(det(\bar{A}))
\end{array}
\right)\left(
\begin{array}{cc}
0 & trop(a)\\
-\infty  & 0
\end{array}
\right)
=trop(P)
trop(L)trop(U)
.$$
\vskip 0.25 truecm

If $trop(A)$ is non-singular over~$R$ then the requirement  $trop(det(A))=det(trop(A))$ is not necessary since the non-singularity of $trop(A)$ will apply; $$trop(1_R)\ne trop(ab)\Rightarrow 1_K\ne ab$$ so  equality must hold between $trop(det(A))$ and $det(trop(A))$. 
\vskip 0.25 truecm

\noindent b. Let $B=(\alpha_{i,j})$. By the same algorithm as in (2.2), we can factor any non-singular matrix over~$R$ and  not strictly singular matrix over $\mathbb{T}$: $P$ will relocate and normalize a dominant monomial of the determinant to the diagonal, and  
$\left(
\begin{array}{cc}
0 & -\infty\\
\beta  & 0
\end{array}
\right)\left(
\begin{array}{cc}
0 & \alpha\\
-\infty  & 0
\end{array}
\right)$ will yield the off diagonal part.

 If $B$ is non-singular over $R$, then $0$  strictly surpasses $\alpha\beta$ and $\alpha\beta+0=0$. If $B$ is  not strictly singular over $\mathbb{T}$ then  $\alpha\beta$ might  equal  $0$, and yet $\alpha\beta+0=0$.

\vskip 0.25 truecm

\noindent c. $B$ is non-singular over $R$, and therefore $det(B)\in T$. Thus  $1_R\ne \alpha\beta$, which means the terms of lowest power of $t$ do not cancel before applying the valuation $trop$. Therefore $1_K\ne trop^{-1}(\alpha)trop^{-1}(\beta)$ and $A$ is invertible. From (2.1) we can conclude that~$trop(P),\ trop(L)$ and $trop(U)$ are the tropicalizations of a product of a diagonal matrix and a permutation matrix, a lower unitriangular matrix and an upper unitriangular matrix, respectively, such that $PLU$ is the factorization of $A$.
\end{proof}

\end{exa}

\section{Nonfactorizable matrices}

\vskip 0.25 truecm

It is important to pay attention to the difference of the factorization process in the post-valuation case and in the pre-valuation case. In matrix theory over a field, the factorization of a matrix is achieved by applying elementary row operations to the matrix in order to transform it to the identity matrix (a process known as Gaussian elimination or reduction of the matrix); then, multiplying the inverses to the corresponding elementary matrices in the opposite order would yield our matrix. In matrix theory over a semifield without  negation, we cannot reduce a nonzero element to zero using elementary operations. Therefore, we are approaching this construction by applying elementary row operations to the identity matrix in order to transform it to our matrix (an expansion of the matrix instead of reduction of the matrix).

\vskip 0.7 truecm

\begin{cla}
For every  elementary matrix $E_1$ of type~1 or~2 and  elementary matrix~$E_2$ of type~3 there exist an elementary matrix~$E_4$ of type~1 or~2 respectively, and an elementary matrix~$E_3$ of type~3, such that~$E_1E_2=E_3E_4$.
\end{cla}

\begin{proof}

This property is well known. We provide the proof here for the reader's convenience.

\noindent Let $E_2=E_{u+k\cdot(m^{th}row)}$ be an elementary matrix of type 3.

\noindent If $E_1=E_{i,j}$ is an elementary matrix~of type 1, then

$$E_1E_2=
\begin{cases}
E_2E_1,\ \text{where}\ u,m\ne i,j\\
E_{j+k\cdot(i^{th}row)}E_{i,j},\ \text{where}\ u=i\ and\ m=j\\
E_{j+k\cdot(m^{th}row)}E_{i,j},\ \text{where}\ u=i\ and\ m\ne j\\
E_{u+k\cdot(i^{th}row)}E_{i,j},\ \text{where}\ u\ne i\ and\ m=j
\end{cases}$$

\vskip 0.25 truecm

\noindent If $E_1=E_{h\cdot(i^{th}row)}$ is an elementary matrix~of type 2, then 

$$E_1E_2=
\begin{cases}
E_2E_1,\ \text{where}\ u,m\ne i\\
E_{i+kh\cdot(m^{th}row)}E_{h\cdot(i^{th}row)},\ \text{where}\ u=i\\
E_{u+\frac{k}{h}\cdot(i^{th}row)}E_{h\cdot(i^{th}row)},\ \text{where}\ m=i
\end{cases}$$

\end{proof}

\noindent Therefore, by symmetry of the last claim, once  a factorization has been obtained, one may construct a factorization whose elementary matrices~of type 3  appear at its ends.  Considering that we are constructing a matrix by applying elementary \textbf{row} operations to the identity matrix, we will be interested throughout the paper in the factorization  whose elementary matrices~of type 3  appear at its left end.

\vskip 0.25 truecm

\vskip 0.25 truecm

\noindent In the next proposition we prove that not every  not strictly singular matrix is factorizable. 

\vskip 0.7 truecm

\begin{pro}
Let $\pi$ and $\sigma$ be two different permutations in $S_n$ such that there exists $t\in \begin{cases}\mathbb{Z}_n,\ where\  n\  is\ odd\\ \mathbb{Z}_n\backslash\{\frac{n}{2}\},\ where\ n\ is\ even\end{cases}$, so that $\pi(i)=\sigma(i)+t\pmod n\ \forall i$ 
\vskip 0.25 truecm

\noindent (i.e. $\pi$ is a shift of $\sigma$, but not by $0$ or $\frac{n}{2}$).
\vskip 0.25 truecm

For $n>2$, any $n\times n$ matrix $A=(a_{i,j})=D_{\pi}+D_{\sigma}$, where $D_{\pi},D_{\sigma}$ are invertible matrices comprised of non-zero permutation tracks $\pi$ and $\sigma$ respectively,  is not factorizable.
\end{pro}

\begin{proof}
\noindent We notice some important facts regarding the process of constructing a factorization for A:

\noindent 1) An elementary row operation of type 3 that changes a $0_R$ entry would raise it beyond adjustment to the entry of $A$, due to the lack of additive inverses.
\vskip 0.25 truecm
\noindent 2) Since the construction starts with the identity matrix and ends with two non-zero permutation tracks, throughout the process, every row and column must have one or two non-zero entries. 
\vskip 0.25 truecm
\noindent 3) Elementary matrices~of types 1 and 2 do not change the number of zeros in the matrix.
\vskip 0.25 truecm
\noindent 4) The requirement $t\ne \frac{n}{2}$ implies that if $\sigma(j)=\pi(i)$ for some $i,j$, then $\sigma(i)\ne\pi(j)$.

\noindent \textit{Proof.} Assume $\sigma (i)=\pi (j)$. Then $$\sigma (j)=\pi(i)= \sigma (i)+t=\pi (j)+t= \sigma (j)+2t\pmod n,$$  which means~$2t=0\pmod n$ and we get~$t=0\ \text{or}\ t=\frac{n}{2}$, contrary to the assumption on~$t$.

\vskip 0.25 truecm

Assume that such a matrix can be factored. According to Claim~3.1 we may obtain a factorization whose elementary matrices~of type 3  appears at its left end. Let us look at the matrix we receive one step before applying this last elementary matrix of type~3. Without loss of generality we may assume it yields the last entry on the track of the permutation~$\sigma$. Therefore we now have a matrix with~$2n-1$ non-zero entries: 
$$a_{i_1,\sigma (i_1)},\cdots,a_{i_{n-1},\sigma (i_{n-1})},\ \ a_{i_1,\pi (i_1)},\cdots,a_{i_{n-1},\pi (i_{n-1})}\text{\and }b,$$ where  $b$ is in the $i_n,\pi(i_n)$  position. We will show that we cannot produce the last non-zero entry under our assumptions, using elementary matrix of type 3.

The last elementary matrix in the factorization would change the zero in the~$i_n,\sigma (i_n)$ position to~$a_{i_n,\sigma (i_n)}$, by adding a row to row~$i_n$. In order to do so, we must use~$a_{k,\pi (k)}$ where~$\pi (k)= \sigma (i_n)$, since it is the only non-zero entry in this column. We already produced the~$a_{k,\sigma (k)}$ entry in the~$k^{th}$ row, which is different than the $k,\pi (i_n)$ position since $\sigma (k)\ne\pi (i_n)$.  This $k,\sigma (k)$ position would influence the~$i_n,\sigma (k)$ position in  row~$i_n$. However, the only other non-zero entry we want in the~$i_n$ row is~$a_{i_n,\pi (i_n)}$. That would require once again~$\sigma (k)=\pi (i_n)$, which cannot occur. 

\begin{center}\includegraphics{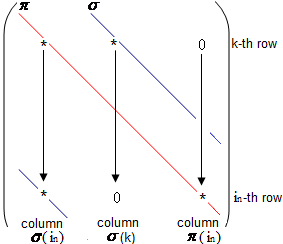}

\tiny{Figure 1. The $k^{th}$ row (on the top) recovers one position in the $i_n^{th}$ row (on the bottom), but, simultaneously, changes a zero-entry beyond adjustment}.\end{center}

\end{proof}

\vskip 0.25 truecm

\begin{exa}

The $3\times 3$ matrix

$$A=\left(
\begin{array}{ccc}
1_R & \alpha_1 & 0_R \\
0_R  &1_R & \alpha_2 \\
\alpha_3 &0_R  & 1_R
\end{array}
\right)$$

\vskip 0.25 truecm

\noindent  is not factorizable, where $\alpha_1,\alpha_2,\alpha_3\ne 0_R$. If it were factorizable then it would have a factorization such that the last elementary matrix~is of type~3. We may assume we have already obtained the first and second rows. The general case is proved analogously by writing $\alpha_{j_i}$ instead of $\alpha_i$, $\forall i=1,2,3$.

\noindent In order to obtain $\alpha_3$ we must use the only non-zero entry in its column, which is in position~$(1,1)$. That is, applying~$E_{3+\alpha_3\cdot (1^{st}row)}$ (the $(1,3)$ position is $0_R$. Therefore the~(3,3) position has already been obtained at this point). However, this operation would raise the zero in the $(3,2)$ position  beyond adjustment.  
\end{exa}

This counterexample provides good intuition for the source of the factorization problem of  tropical matrices. In a way, some entries are "too small" to be obtained by any factorization. In fact, this example will function as the base case for the inductive proof of the classification of factorization of $3\times 3$ matrices.
\vskip 0.25 truecm

These non-factorizable matrices relate to Buchholz' conditions for multiplicativity of the~$trop$ valuation in the following way: 

\begin{lem} Let $A,B$ be square matrices of size $n$ over the field $K=\mathbb{C}\{\{t\}\}$ of Puiseux series. Then \begin{equation}trop(XY)=trop(X)trop(Y),\text{ where }X=(a_{i,j}),Y=(b_{i,j}),\end{equation} if and only if the terms of lowest power of $t$ in the Puiseux series $\Sigma_{k=1}^na_{i,k}b_{k,j}$ do not cancel for every $i$ and $j$.
\end{lem}

\begin{proof} See ~\cite[Lemma 4.36]{TLI}.\end{proof}

Looking at  $$B=trop^{-1}(A)=
\left(
\begin{array}{ccc}
1_K & a_1 & 0_K \\
0_K  &1_k & a_2 \\
a_3 &0_K  & 1_K
\end{array}
\right),\ where\ a_i=trop^{-1}(\alpha_i),$$
the matrix $trop^{-1}(A)$ can be factored as
$$LU=\left(
\begin{array}{ccc}
1_K &  0_K & 0_K \\
0_K  &1_k &  0_K \\
a_3 &-a_1a_3 & 1_K
\end{array}
\right)\left(
\begin{array}{ccc}
1_K & a_1 & 0_K \\
0_K  &1_k & a_2 \\
 0_K &0_K  & 1_K+ a_1 a_2 a_3
\end{array}
\right).$$
Next, $trop(LU)=A$ while $trop(L)trop(U)$ will yield
$$\left(
\begin{array}{ccc}
1_R &  0_R & 0_R \\
0_R &1_R &  0_R \\
\alpha_3 & \alpha_1\alpha_3 & 1_R
\end{array}
\right)\left(
\begin{array}{ccc}
1_R & \alpha_1 & 0_R \\
0_R  &1_R & \alpha_2 \\
 0_K &0_K  & 1_R+ \alpha_1 \alpha_2 \alpha_3
\end{array}
\right).$$

\noindent Looking at position~$(3,2)$ in these two matrices, we notice that~$(trop(L)trop(U))_{3,2}$ is~$\alpha_1\alpha_3+\alpha_1\alpha_3$, while~$(A)_{3,2}$ is $-\infty$, which means the terms of lowest power of $t$ in the Puiseux series $(trop^{-1}(A))_{3,2}$ are  canceled, creating  inequality in (3.1).

\vskip 0.7 truecm

\section{Factorization of  $3\times 3$  matrices}

The fact that the determinant of a non-singular matrix $A$ over $R$ is tangible means that the matrix has one dominant permutation track. By using elementary matrices of type 1 (a permutation matrix) we can relocate the corresponding permutation to the diagonal and by using elementary matrices of type 2 (a diagonal matrix) we can change the diagonal entries to~$1_R$, receiving a non-singular matrix with dominant Id-permutation track equals~$1_R$. That is, $A=P\bar{A}$ where $P$ is an invertible matrix (See Remark 2.7) such that $|P|=|A|$ and~$|\bar{A}|=1_R$. 

We denote $\bar{A}$ as the \textbf{ normal form of $A$}, and say that $P$ normalizes the dominant permutation track to the diagonal. This is not the same as normal matrices, defined in ~\cite{TLI}.

We may also obtain a normal form for not strictly singular matrices by relocating and normalizing one of the dominant permutation tracks. However, such a matrix will  have an invertible determinant over $\mathbb{T}$ and $T$. Therefore, by "matrix in normal form" we mean normal forms of not strictly singular matrices over $\mathbb{T}$, or normal forms of non-singular matrices over $R$.

\begin{rem} If $\bar{A}$ is in normal form, then every permutation track is dominated by $1_R$ (with the possibility of equality when we are working over~$\mathbb{T}$). 

Often, for each permutation track, we write the permutation as a product of disjoint cycles, referred to as the cycle tracks of the permutation track. 

We notice that since the entries on the diagonal are $1_R$, every cycle track itself presents a permutation track, when composed with the appropriate Id cycle tracks. Therefore every cycle track is  being dominated by $1_R$ (with the possibility of equality when we are working over~$\mathbb{T}$). Consequently,  for any term $M$ we get that~$M\cdot$(cycle track)  is  dominated by $M$ (with the possibility of equality when we are working over~$\mathbb{T}$).
\end{rem}

\begin{cla}
A non-singular matrix $A$ over $R$ (or not strictly singular over $\mathbb{T}$) is factorizable if and only if its normal form $\bar{A}$ is factorizable.
\end{cla}

\begin{proof}

\noindent Let $P$ be the invertible matrix that normalizes the dominant permutation track in $A$ to the diagonal.  According to Remark~2.7 $P$ is invertible, and we can conclude that~$\bar{A}=P^{-1}A$. Of course $P^{-1}$ is invertible, and therefore by using Remark~2.7 again we have that $P^{-1}$ is also a product of elementary matrices~of type~1 and 2. Hence the factorizability of~$A$ and the factorizability of~$\bar{A}$ are equivalent.
\end{proof}

\begin{lem}
 
Given any nondiagonal entry $a_{i,j}$ of a $3\times 3$ matrix, there exists precisely one permutation track in which this nondiagonal entry appears  and for which all other entries are also nondiagonal.
\end{lem}
\vskip 0.2 truecm

\begin{proof}
 A permutation track is of the form $a_{1,\pi (1)}a_{2,\pi (2)}a_{3,\pi (3)}$ for some $\pi \in S_3$. If we want all the entries to be nondiagonal, then $\pi(i)\ne i,\ \forall i=1,2,3$, which means $\pi=(1\ 2\ 3)$ or~$(1\ 3\ 2)$ and therefore the possible  permutation tracks are $a_{1,2}a_{2,3}a_{3,1}$ and $a_{1,3}a_{3,2}a_{2,1}$ which consist of \textit{all} of the nondiagonal entries,  \textit{exactly} one time each.
\end{proof}

\vskip 0.25 truecm

\begin{df} An \textbf{entry condition} is the relation ($>,<,=$) between a nondiagonal entry $a_{i,j}$ in  normal form $A=(a_{i,j})$ and the product $a_{i,k}a_{k,j}$ of the other nondiagonal entry in its row and the other nondiagonal entry in its column. We refer to the matrix 
$$(b_{i,j})\text{ where }b_{i,j}=\begin{cases}1_R,\text{ where } i=j\\a_{i,j}+a_{i,k}a_{k,j},\text{ where }i,j\text{ and }k\text{ are distinct }\end{cases}$$
as \textbf{the matrix of entry conditions}.
\end{df}

\begin{lem}
A $3\times 3$ matrix in normal form over $R$  is not factorizable if and only if there exists a permutation track of  nondiagonal entries all of whose entry conditions satisfy~$<$.  
\end{lem}

\begin{proof}

\noindent We denote $A=(a_{i,j})$. 

\noindent \underline{$(\Rightarrow$)} There exists a row~$j$ such that both of its entry conditions satisfy~$\geq$. 

We can obtain any $2\times 2$ minor, by using the algorithm in Example 2.9-part(b), as~$2\times 2$ matrices embedded to  rows and columns $i_1$ and $i_2$ in  the $3\times 3$ identity matrix. Then we can recover the third column $i_3$, using $1_R$ in the $i_3,i_3$ position, obtaining a~$2\times 3$ minor:
 $$\underbrace{E_{i_2+a_{i_2,i_3}\cdot(i_3^{th}row)}\cdot E_{i_1+a_{i_1,i_3}\cdot(i_3^{th}row)}}_\text{Recovering the third column}\cdot\  \underbrace{E_{i_2+a_{i_2,i_1}\cdot(i_1^{th}row)}\cdot E_{i_1+a_{i_1,i_2}\cdot(i_2^{th}row)}}_\text{Obtaining a $2\times 2$ minor}.$$

\vskip 0.2 truecm

\noindent Therefore we can conclude that obtaining the third row $i_3$ is the only obstruction to the factorization process of a $3\times 3$ nonsingular matrix.

\vskip 0.25 truecm

\noindent Since $a_{i,j}a_{j,i}<1_R$ for every $i,j$ such that $i\ne j$, we get that if  some nondiagonal entry condition is $<$: $a_{i,j}<a_{i,k}a_{k,j}$, then the entry conditions of the nondiagonal entries in its row and column are $>$: $$a_{i,j}(a_{j,k})<a_{i,k}(a_{k,j}a_{j,k})<a_{i,k}\text{ and }(a_{k,i})a_{i,j}<(a_{k,i}a_{i,k})a_{k,j}<a_{k,j}.$$

\vskip 0.25 truecm

 Therefore, we may assume that the row remains to be constructed is row number three. The general case is being proved analogously by writing $a_{i_t,j}$ instead of~$a_{t,j}$ for~$t=1,2,3$.

\vskip 0.25 truecm

\noindent If both entry conditions are $>$, then by applying $E_{3+a_{3,i}\cdot(i^{th}row)}$ for~$\ i=1,2$ to the $2\times 3$ minor of rows one and two, we obtain $A$:
  
$ $

$$\left(
\begin{array}{ccc}
1_R & 0 & 0\\
0 & 1_R & 0\\
a_{3,1} & 0 & 1_R
\end{array}
\right)
\left(
\begin{array}{ccc}
1_R & 0 & 0\\
0 & 1_R & 0\\
0 & a_{3,2} & 1_R
\end{array}
\right)
\left(
\begin{array}{ccc}
1_R & a_{1,2} & a_{1,3}\\
a_{2,1} & 1_R & a_{2,3}\\
0 & 0 & 1_R
\end{array}
\right)=$$
$$=\left(
\begin{array}{ccc}
1_R & a_{1,2} & a_{1,3}\\
a_{2,1} & 1_R & a_{2,3}\\
a_{3,1}+a_{3,2}a_{2,1} & a_{3,2}+a_{3,1}a_{1,2} & 1_R
\end{array}
\right)=
\left(
\begin{array}{ccc}
1_R & a_{1,2} & a_{1,3}\\
a_{2,1} & 1_R & a_{2,3}\\
a_{3,1} & a_{3,2} & 1_R
\end{array}
\right).$$ 
 
If  one entry condition satisfies $=$ then  we apply only the operation that does not correspond to the $=$ entry condition. Both satisfy $=$ cannot occur since that means: $$a_{i,j}=a_{i,k}a_{k,j}=a_{i,j}a_{j,k}a_{k,j}<a_{i,j}.$$ Contradiction. 
\vskip 0.25 truecm

\noindent \underline{$(\Leftarrow)$} If there is no row whose entry conditions both satisfy $\geq$, then each row and column has exactly one $<$ condition (two conditions $<$ cannot occur in the same row or column). As a result, we get one nondiagonal permutation track of  entry conditions~$>$ and the other nondiagonal permutation track of  entry conditions $<$. We show that in this case the normal form $A$ is not factorizable.  

\vskip 0.25 truecm

Assume $A$ is factorizable. By Claim 3.1 we can obtain a factorization whose elementary matrices~of type 3  appears at its left end. We now look at the last elementary matrix~of type $3$ in this factorization (which is the last elementary matrix~in this factorization). As before, we assume that this elementary matrix~of type 3 operates on the third row, meaning, we are starting with
\begin{equation}\left(
\begin{array}{ccc}
1_R & a_{1,2} &  a_{1,3}\\
 a_{2,1} & 1_R & a_{2,3}\\
\alpha_1 & \alpha_2 & \alpha_3
\end{array}
\right)\end{equation}
\noindent where at least one entry in the third row is different from the entry we want to produce. By applying elementary matrix of type 3, we aspire to obtain the third row:~$( a_{3,1}\   a_{3,2}\ 1_R)$.
\vskip 0.25 truecm
\noindent Assume $a_{3,j}<a_{3,k}a_{k,j}$ (which means $a_{3,k}>a_{3,j}a_{j,k}$). By definition, $j,k$ and $3$ are different and we may apply $ E_{3+m_j\cdot(j^{th}row)}$ or $ E_{3+m_k\cdot(k^{th}row)}$.
\vskip 0.25 truecm

\noindent In case we applied $ E_{3+m_j\cdot(j^{th}row)}$ we get $$a_{3,j}=\alpha_j+m_j\Rightarrow m_j\leq a_{3,j}\Rightarrow \begin{cases}m_ja_{j,k}\leq a_{3,j}a_{j,k}<a_{3,k}\Rightarrow a_{3,k}=\alpha_k+m_ja_{j,3}=\alpha_k\\m_ja_{j,3}\leq a_{3,j}a_{j,3}<1_R\Rightarrow       a_{3,3}=1_R=\alpha_3+m_ja_{j,3}=\alpha_3\end{cases}.$$ 

\noindent In case we applied $ E_{3+m_k\cdot(k^{th}row)}$ we get $$a_{3,j}=\alpha_j+m_ka_{k,j}\Rightarrow  m_ka_{k,j}\leq a_{3,j}<a_{3,k}a_{k,j}\Rightarrow a_{3,k}>m_k\Rightarrow \begin{cases} a_{3,k}=\alpha_k+m_k=\alpha_k\\a_{3,3}=1_R=\alpha_2+m_ka_{k,3}=\alpha_3\end{cases}.$$

Hence, the matrix in (4.1)  differ from $A$  by one entry, the one of condition $<$ in the third row.

\vskip 0.25 truecm

If the remaining entry is bigger than the  desired one then clearly we cannot produce the desired entry with an elementary matrix~of type 3. Thus, it must be smaller, which means we are back to the same entry conditions: one permutation track of nondiagonal entry conditions $<$. Therefore, looking at all the previous elementary matrices~of type~3, we notice that it would again yield a matrix, changed only at the permutation track of nondiagonal entry conditions $<$,  by reducing them:

$ $

$$\left(
\begin{array}{ccc} 
1_R & c &a_{1,3}\\
a_{2,1} & 1_R & b\\
a & a_{3,2} & 1_R
\end{array}
\right)\ 
\text{or} \  
\left(
\begin{array}{ccc}
1_R & a_{1,2} & c\\
a & 1_R & a_{2,3}\\
a_{3,1} & b & 1_R
\end{array}
\right)$$
\begin{center}\tiny{$a,b$ and $c$ has entry conditions $<$.}\end{center}

\noindent If $a,b$ or $c$ never reaches $0_R$ in the string of elementary matrices~of type $3$, then the factorization does not terminate, which is not possible. That leads us to the conclusion that $a,b$ and $c$ reach $0_R$ in this string, which means at some point of the factorization we get either
 
$$\left(
\begin{array}{ccc} 
1_R & 0_R &a_{1,3}\\
a_{2,1} & 1_R & 0_R\\
0_R & a_{3,2} & 1_R
\end{array}
\right)\ 
\text{or} \  
\left(
\begin{array}{ccc}
1_R & a_{1,2} & 0_R\\
0_R & 1_R & a_{2,3}\\
a_{3,1} & 0_R & 1_R
\end{array}
\right)$$

\noindent where the nondiagonal entries remained, are strictly bigger than~$0_R$ due to their entry conditions which satisfy $>$. These matrices are not factorizable according to Proposition~3.2 and therefore cannot appear as a part of any factorization. 

\end{proof}

In Buchholz' terminology, an entry condition $<$ implies the connection between the desired entry, $a_{i,j}$,  and the entry obtained by factorizing over $K$ and then tropicalizing each component. The lowest power of $t$ in $trop^{-1}(a_{i,j})$  is being canceled, causing the valuation of this entry to rise and the tropicalization to drop.

\vskip 0.4 truecm

\begin{exa}

 $$A=
\left(
\begin{array}{ccc} 
0 & -3 & 0\\
1 & 5 & 0\\
3 & 1 & 6
\end{array}
\right)$$ 
\noindent We easily calculate that $|A|=11$ and that the diagonal is the dominant permutation track. Therefore, by applying $E_{5\cdot (2^{nd}row)}E_{6\cdot(3^{rd}row)}$ to the normal form $$\bar{A}=
 \left(
\begin{array}{ccc} 
0 & -3 & 0\\
-4 & 0 & -5\\
-3 & -5 & 0
\end{array}
\right)$$
we might achieve a factorization of $A$.

\noindent The next step would be to check the entry conditions of $\bar{A}$, which may be displayed as:

$$ \left(
\begin{array}{ccc} 
0 & > & >\\
> & 0 & <\\
> & > & 0
\end{array}
\right).$$
We do not have a permutation track of entry conditions $<$ (as could be seen after checking the first row) and therefore $A$ is factorizable, indeed:

$$A=
  \underbrace{\left(
\begin{array}{ccc} 
0 & - & -\\
- & 5 & -\\
- & - & 0
\end{array}
\right)\cdot
 \left(
\begin{array}{ccc} 
0 & - & -\\
- & 0 & -\\
- & - & 6
\end{array}
\right)}_\text{$\bar{A}\rightarrow A$}\cdot
 \underbrace{ \left(
\begin{array}{ccc} 
0 & - & -\\
- & 0 & -\\
- & -5 & 0
\end{array}
\right)\cdot
 \left(
\begin{array}{ccc} 
0 & - & -\\
- & 0 & -\\
-3 & - & 0
\end{array}
\right)}_\text{$3^{rd}\ row$}\cdot$$
$$ \underbrace{\left(
 \begin{array}{ccc} 
0 & - & -\\
- & 0 & -5\\
- & - & 0
\end{array}
\right)\cdot
 \left(
\begin{array}{ccc} 
0 & - & 0\\
- & 0 & -\\
- & - & 0
\end{array}
\right)}_\text{$2\times 3$}\cdot
  \underbrace{\left(
\begin{array}{ccc} 
0 & -3 & -\\
- & 0 & -\\
- & - & 0
\end{array}
\right)\cdot
 \left(
\begin{array}{ccc} 
0 & - & -\\
-4 & 0 & -\\
- & - & 0
\end{array}
\right)}_\text{$2\times 2$}.$$
\end{exa}

\vskip 0.5 truecm

\begin{exa}
 $$A=
\left(
\begin{array}{ccc} 
4 & 3 & 3\\
4 & 5 & 2\\
5 & 7 & 6
\end{array}
\right)$$ 
\noindent We easily calculate that $|A|=15$ and that the diagonal is the dominant permutation track. Therefore, by applying $E_{4\cdot (1^{st}row)}E_{5\cdot (2^{nd}row)}E_{6\cdot (3^{rd}row)}$ to the normal form 
$$\bar{A}=
 \left(
\begin{array}{ccc} 
0 & -1 & -1\\
-1 & 0 & -3\\
-1 & 1 & 0
\end{array}
\right)$$
we might achieve a factorization of $A$.

\noindent The next step would be to check the entry conditions of $\bar{A}$, which may be displayed as:

 $$ \left(
\begin{array}{ccc} 
0 & {{<}} & > \\
> & 0 &  {{<}}\\
 {{<}} & > & 0
\end{array}
\right).$$

\noindent We have a permutation track of entry conditions $<$. Therefore, looking at all the previous elementary matrices~of type 3 would yield the matrix

$$\left(
\begin{array}{ccc} 
0 & a & -1\\
-1 & 0 & b\\
c & 1 & 0
\end{array}
\right)$$

\noindent where either $a,b,c=0_R$, which is not a factorizable matrix, or the factorization does not terminate. Therefore $A$ is not factorizable. 

\end{exa}

This classification would be rather hard to generalize for $n\times n$ matrices since the required number of conditions  increases significantly. In the next section we present a~$\nu$-equivalent approach to supertropical matrices that  helps us in constructing a general tropical factorization for non-singular matrices over $R$ and not strictly singular matrices over $\mathbb{T}$.

\section{$3\times 3$ Quasi-factorization}

In order to recover a factorization result for  not strictly singular $3\times 3$ matrices, we follow the terminology in ~\cite{STLA} when extending the classical definitions by considering the supertropical ghost ideal.

\begin{df} A \textbf{quasi-zero} matrix $Z_G$ is a matrix equal to $0_R$ on the diagonal, and whose off-diagonal entries are ghosts or $0_R$.  A \textbf{quasi-identity} matrix $I_G$ is a nonsingular, multiplicatively idempotent matrix equal to $I + Z_G$, where $Z_G$ is a quasi-zero matrix. %A matrix $B$ is a \textbf{quasi-inverse} for $A$ if $AB$ and $BA$ are quasi-identities.
\end{df}

\begin{df} The $\mathbf{t,l}$-\textbf{minor} $A_{t,l}$ of a matrix $A = (a_{i,j})$ is obtained by deleting the $t^{th}$ row and $l^{th}$ column of $A$. The \textbf{adjoint matrix} $adj(A)$ of $A$ is defined as the matrix~$(a'_{i,j} )$, where
$a'_{i,j} =det(A_{j,i})$. The matrix $A^{\nabla}$ denotes $\frac{adj(A)}{det(A)}$, when $det(A)$ is invertible. Over $R$, $A^{\nabla}$ is defined for non-singular matrices only. Over $\mathbb{T}$, however, $A^{\nabla}$ is defined for every not strictly singular matrix. 

\noindent Notice that $det(A_{j,i})$ may be observed as the sum of all permutation tracks in $A$ that passes through $a_{j,i}$, which is then  deleted from these permutation tracks:    $$det(A_{j,i})=\sum_{\tiny{\begin{array}{cc}\sigma\in S_n:\\\sigma(j)=i\end{array}}}a_{1,\sigma(1)}\cdots a_{j-1,\sigma(j-1)}a_{j+1,\sigma(j+1)}\cdots a_{n,\sigma(n)}.$$ When writing  such a permutation as the product of its disjoint cycles, $det(A_{j,i})$ can be presented as: $$det(A_{j,i})=\sum_{\tiny{\begin{array}{cc}\sigma\in S_n:\\\sigma(j)=i\end{array}}}(a_{i,\sigma(i)}\cdots a_{\sigma^{-1}(j),j})C_{\sigma},$$ where $(a_{i,\sigma(i)}\cdots a_{\sigma^{-1}(j),j})$ is the cycle track missing $a_{j,i}$,  and $C_{\sigma}$ is the product of the cycle tracks in $\sigma$ that do not include $i$ and $j$.

\end{df}

\begin{df}  For $det(A)$  invertible, we say that $B$ is a \textbf{quasi-inverse} of $A$ over $R$ if $AB=I_G$ and~$BA= I'_G$ where $I_G,I'_G$ are quasi-identities.\end{df}

\begin{lem}$ $

(i) $A^{\nabla}$ is a quasi-inverse of $A$.

(ii) $A$ is a quasi-inverse of $A^{\nabla}$.
\end{lem}
 \begin{proof} See ~\cite[Theorem 2.8]{STMA2} \end{proof}

\noindent We denote $I_A = AA^{\nabla}$ and $I'_A=A^{\nabla}A$, which are quasi-identity matrices. 
\vskip 0.5 truecm

Notice that for $A$ in normal form $A^{\nabla}=\frac{adj(A)}{det(A)}=adj(A)$, which is also in normal form: the diagonal entries in~$adj(A)$ are  sums of cycle tracks of $A$, thus the Id summand $1_R$ dominates every diagonal entry. Also, $$1_R=det(I_A)=det(AA^{\nabla})=det(A)det(A^{\nabla})=det(A^{\nabla})\text{ (since }det(AA^{\nabla})\in T),$$ as required for  normal form. 

\begin{thm} $ $

(i) $det(A\cdot adj(A))= det(A)^n$ .

(ii) $det(adj(A))= det(A)^{n-1}$ .
\end{thm}

\noindent \textbf{Proof.} ~\cite[Theorem 4.9]{STMA}.

\begin{pro} $ adj(AB) \vDash_{gs} adj(B) adj(A)$.\end{pro}

\noindent \textbf{Proof.} ~\cite[Proposition 4.8]{STMA}.

\begin{lem}
 
$ $

\noindent (i) $P^\nabla =P^{-1}$ whenever~$P$ is an invertible matrix. 

\noindent (ii) $(PA)^\nabla =A^\nabla P^\nabla$ where~$det(A)$ is invertible and~$P$ is an invertible matrix.

\noindent (iii) Let~$\bar{A}$ be the normal form of the matrix $A$ (i.e.~$A=P\bar{A}$ where~$P$ is the invertible matrix that normalizes a dominant permutation track of~$A$ to the diagonal). 

\noindent Then~$A^\nabla =\bar{A}^{\nabla}P^{-1}$.
\end{lem}

\begin{proof} 
$ $

\noindent (i) Let $P$ be an invertible matrix of order $n$. According to Remark 2.7 $P$ is a product of a diagonal matrix $D=(d_i)$ (i.e. has $d_i$ in the $i,i$ position, and $0_R$ otherwise), and a permutation matrix $P_{\pi}=\sum_{i=1}^n e_{i,\pi(i)}$, where $e_{i,j}$ is the matrix with $1_R$ in the $i,j$ position and $0_R$ otherwise. Therefore$$P=DP_{\pi}=\sum_{i=1}^n d_ie_{i,\pi(i)}.$$ By definition $$P^{\nabla}=\frac{1}{\prod_{i=1}^n d_i}\sum_{i=1}^n \left(\prod_{j\ne i}d_je_{\pi(i),i}\right),$$ and we can conclude $P^{\nabla}=P^{-1}$ from $$PP^{\nabla}=P^{\nabla}P=\sum_{i=1}^n e_{i,i}=I.$$ 

\noindent (ii) Using part (i) and the well known fact that~$(AB)^{-1}=B^{-1}A^{-1}$, for invertible matrices~$A,B$, we get $$(DP_{\pi})^{\nabla}=P^{\nabla}=P^{-1}=(DP_{\pi})^{-1}=P_{\pi}^{-1}D^{-1}=P_{\pi}^{\nabla}D^{\nabla},$$ where $D=(d_i)$ is a diagonal matrix, $P_{\pi}=\sum_{i=1}^n e_{i,\pi(i)}$ is a permutation matrix, and $P$ is the invertible matrix composed by $D$ and $P_{\pi}$. 

We denote $A$ as $(a_{i,j})$ and show how $P$ acts on $A$;
$$PA=DP_{\pi}(a_{i,j})=D(a_{i,\pi(j)})=(d_ia_{i,\pi(j)}).$$ Since $P$ is invertible we have $$det(A)=det(P^{-1}PA)\vDash_{gs}det(P^{-1})det(PA)\vDash_{gs} det(P)^{-1}det(P)det(A)=det(A),$$ which means  $$det(P^{-1})det(PA)=det(P)^{-1}det(P)det(A).$$ Thus $det(PA)=det(P)det(A)$ and we get
$$(PA)^\nabla=\frac{1}{\prod d_i\cdot det(A)}\left(\sum_{\sigma(\pi(j))=i} (d_ia_{i,\sigma(i)}\cdots d_{\sigma^{-1}(\pi(j))}a_{\sigma^{-1}(\pi(j)),\pi(j)})C_{\sigma}\right).$$
Each summand in  the  numerator   includes $d_k$ for every $k\ne \pi(j)$, therefore $$(PA)^\nabla=\frac{1}{\prod d_i\cdot det(A)}\left(\sum_{\sigma(\pi(j))=i} \prod_{k\ne\pi(j)} d_k(a_{i,\sigma(i)}\cdots a_{\sigma^{-1}(\pi(j)),\pi(j)})C_{\sigma}\right)$$$$=\frac{1}{ det(A)}\left(\sum_{\sigma(\pi(j))=i} (a_{i,\sigma(i)}\cdots a_{\sigma^{-1}(\pi(j)),\pi(j)})C_{\sigma}\right)D^\nabla$$$$=\frac{1}{ det(A)}\left(\sum_{\sigma(j)=i} (a_{i,\sigma(i)}\cdots a_{\sigma^{-1}(j),j})C_{\sigma}\right)P_{\pi}^\nabla D^\nabla=A^\nabla P^\nabla.$$

\noindent (iii) Using the arguments in (i) and (ii) we have $$A^\nabla=(P\bar{A})^\nabla=\bar{A}^\nabla P^\nabla=\bar{A}^\nabla P^{-1}$$ as required.

\end{proof}

This last result allows us to approach the factorization in two stages, preserving the \textgravedbl well-behaved\textacutedbl $ $ invertible part, consist of elementary matrices~of type $1$ and $2$, and aspire to obtain factorization of the remaining, quasi-invertible part, including elementary matrices~of type 3.  

In fact, in his work, Buchholz salvage  a positive answer for the multiplicativity of the tropicalization where $A$ or $B$ are permutation or diagonal matrices. This result, in the post-valuation setting, will pass the factorization problem to the non-invertible matrices, the normal forms. 

The following lemma describes the solution by passing to matrices that are, in a supertropical way, equivalent to the original ones.

\begin{lem}The following assertions hold for any $3\times 3$ matrix $A$ in normal form. 

\noindent 1. $A^{\nabla}$ is the matrix of entry conditions defined in 4.4.

\noindent 2. $A^{\nabla}$ is always factorizable.

\noindent 3. $A^{\nabla\nabla}=A^{\nabla}$. (Equality holds over $\mathbb{T}$, and is being interpreted over $R$ as $\cong_{\nu}$).

\noindent 4. $A^{\nabla\nabla}$ is always factorizable.
\end{lem}

\begin{proof} By hypothesis, writing~$A=(a_{i,j})$ where~$a_{i,j}=1_R$ for~$i=j$, we have~$$1_R\geq a_{i,j}a_{j,i}\text{ and }1_R\geq a_{i,k}a_{k,j}a_{j,i}$$ for every distinct~$i,j$ and~$k$.

\noindent 1. By definition of $A^{\nabla}=(a'_{i,j})$ and the matrix of entry conditions,$ (b_{i,j})$, we get:
$$a'_{i,j}=
\begin{cases}
1_R,\  i=j\\
\sum(a_{i,\sigma(i)}\cdots a_{\sigma^{-1}(j),j})C_{\sigma},\  i\ne j
\end{cases}=
\begin{cases}
1_R,\  i=j\\
a_{i,j}+a_{i,k}a_{k,j},\  i\ne j
\end{cases}=
b_{i,j}$$

\vskip 0.5 truecm

\noindent 2. According to Lemma 4.5 we can factor the matrix iff there exists a row with two entry conditions $\geq$. For every entry in $A^{\nabla}=(b_{i,j})$ we have $$b_{i,k}b_{k,j}=(a_{i,k}+a_{i,j}a_{j,k})(a_{k,j}+a_{k,i}a_{i,j})=a_{i,k}a_{k,j}+(a_{i,k}a_{k,i})a_{i,j}+a_{i,j}(a_{j,k}a_{k,j})+a_{i,j}(a_{j,k}a_{k,i}a_{i,j}).$$ Since every term in brackets is a cycle track we get $$b_{i,k}b_{k,j}\leq a_{i,j}+a_{i,k}a_{k,j}=b_{i,j}.$$

\vskip 0.5 truecm

\noindent 3. By definition
$$A^{\nabla\nabla}=(c_{i,j}),\text{ where }c_{i,j}=
\begin{cases}
1_R,\  i=j\\
b_{i,j}+b_{i,k}b_{k,j},\  i\ne j
\end{cases}=
b_{i,j}$$ 

\vskip 0.5 truecm

\normalsize

\noindent 4. Immediate from (2) and (3).
\end{proof}

In the next section we  generalize parts (2)-(4) of the last Lemma for not strictly singular $n\times n$ matrices.

\section{$n\times n$ Quasi-factorization}

Having established factorizability for the quasi-inverses of not strictly singular~$3\times 3$  matrices, we would like to obtain this result for $n\times n$ not strictly singular matrices. In order to do so we achieve the same equality over~$\mathbb{T}$ as in  Lemma 5.8- part 3, which  guarantees that the  ghost value of the entries of a matrix in normal form over $R$ are preserved under $\nabla$. In view of Lemma 5.7 we may assume that every not strictly singular matrix $A$ is in normal form.

\begin{cla} If $A$ is in normal form, then $A^\nabla A= A^\nabla= AA^\nabla$. (Equality holds over~$\mathbb{T}$, and is being interpreted over $R$ as $\cong_{\nu}$).\end{cla}

\begin{proof}

$AA^\nabla,\ A^\nabla A$ are quasi-identities and therefore equal to  $1_R$  on the diagonal, as is~$A^\nabla$. We check $AA^\nabla$ outside the diagonal. The proof for $A^\nabla A$ would be obtained analogously by exchanging $a_{i,k}a'_{k,j}$ by $a'_{i,k}a_{k,j}$, where $A=(a_{i,j}),\ A^\nabla =(a'_{i,j})=(det(A_{j,i}))$. 

 For $i\ne j$, the $i,j$ entry of $AA^\nabla$  is a sum of the form 

\begin{equation}\sum_{k=1}^n a_{i,k}a'_{k,j}=a_{i,i}a'_{i,j}+a_{i,j}a'_{j,j}+\sum_{k\ne i,j}a_{i,k}a'_{k,j}.\end{equation}

\noindent Since $A^\nabla, A$ are in normal form their diagonal entries are $1_R$, yielding

\begin{equation}a'_{i,j}+a_{i,j}+\sum_{k\ne i,j}a_{i,k}a'_{k,j}.\end{equation}

Clearly $A A^\nabla\geq A^\nabla,$ since $a'_{i,j}$ is a summand in the $i,j$ position of $A A^\nabla$. So it suffices to prove that $A^\nabla  \geq AA^\nabla$. Moreover, we saw $$a'_{i,j}=\sum_{\tiny{\begin{array}{cc}\pi\in S_n:\\\pi(j)=i\end{array}}}(a_{i,\pi(i)}\cdots a_{\pi^{-1}(j),j})C_{\pi}=a_{i,j}+\sum_{\tiny{\begin{array}{ccc}\pi\in S_n:\\\pi(j)=i\\\pi\ne(i\ j)\end{array}}}(a_{i,\pi(i)}\cdots a_{\pi^{-1}(j),j})C_{\pi},$$ therefore $a'_{i,j}+a_{i,j}$ in (6.2) is $a'_{i,j}$.
\vskip 0.25 truecm

\noindent By definition \begin{equation}a_{i,k}a'_{k,j}=a_{i,k}\sum_{\tiny{\begin{array}{cc}\sigma\in S_n:\\\sigma(j)=k\end{array}}}(a_{k,\sigma(k)}\cdots a_{\sigma^{-1}(j),j})C_{\sigma}\end{equation} and \begin{equation}a'_{i,j}=\sum_{\tiny{\begin{array}{cc}\pi\in S_n:\\\pi(j)=i\end{array}}}(a_{i,\pi(i)}\cdots a_{\pi^{-1}(j),j})C_{\pi}.\end{equation}

\noindent We will show that (6.4) is equal to or greater than every summand in (6.3) for every~$k$, by finding for every permutation $\sigma$ in (6.3) a permutation $\pi$ in (6.4) such that the summand of $\pi$ dominates the summand of $\sigma$.

\vskip 0.25 truecm

Considering Remark (4.1), it suffices to show we can factor each summand in (6.3) into a cycle track, which is smaller than $1_R$, and a summand in (6.4). 

For every $k\ne i,j$ and $\sigma \in S_n$ such that $\sigma(j)=k$ we look at $$a_{i,k}(a_{k,\sigma(k)}\cdots a_{\sigma^{-1}(j),j})C_{\sigma}$$ and distinguish between the two cases, whether or not  $i$ is in the same  cycle as $j$ in $\sigma$.

\noindent \underline{Case I:}   $i$ is in the same  cycle as $j$: \begin{equation}a_{i,k}(a_{k,\sigma(k)}\cdots a_{\sigma^{-1}(i),i}a_{i,\sigma(i)}\cdots a_{\sigma^{-1}(j),j})C_{\sigma},\end{equation} where $C_{\sigma}$ is the product of the remaining cycle tracks in $\sigma$. By factoring this expression into the following disjoint terms: $(a_{i,k}a_{k,\sigma(k)}\cdots a_{\sigma^{-1}(i),i})C_{\sigma}$ and $(a_{i,\sigma(i)}\cdots a_{\sigma^{-1}(j),j})$, and composing each with disjoint Id tracks, we obtain a permutation track and a summand in (6.4) respectively.

\noindent Consequently it is clear that these summands in (6.3) are  dominated by (6.4).

\vskip 0.25 truecm

\noindent \underline{Case II:}   $i$ is not in the same  cycle as $j$: \begin{equation}a_{i,k}(a_{k,\sigma(k)}\cdots a_{\sigma^{-1}(j),j})(a_{i,\sigma(i)}\cdots a_{\sigma^{-1}(i),i})C'_{\sigma},\end{equation}where $C'_{\sigma}$ is the product of the remaining  cycle tracks in $\sigma$. By factoring this expression into the disjoint terms:~$a_{i,k}(a_{k,\sigma(k)}\cdots a_{\sigma^{-1}(j),j})$ and $(a_{i,\sigma(i)}\cdots a_{\sigma^{-1}(i),i})C'_{\sigma}$ , and composing each with disjoint Id tracks, we obtain  a summand in~(6.4)  and  a permutation track respectively. Consequently it is clear that these summands in (6.3) are also dominated by (6.4).

\vskip 0.25 truecm

Hence, we get:
$$a'_{i,j}+a_{i,j}+\sum_{k\ne i,j}a_{i,k}a'_{k,j}=a'_{i,j}.$$

\end{proof}

\vskip 0.25 truecm

\begin{cor} Let $A$ be a matrix with  normal form $\bar{A}$, i.e. $A=P\bar{A}$ for some invertible matrix $P$. Then

a. $\bar{A}^{\nabla\nabla} = \bar{A}^\nabla$.

b. $A^{\nabla\nabla} = PA^\nabla P$.

\noindent (Equalities hold over $\mathbb{T}$, and are being interpreted over $R$ as $\cong_{\nu}$)
\end{cor}

\begin{proof}

$ $

a. According to ~\cite[Corollary 4.4]{STMA2} we know that $A^\nabla =A^\nabla A^{\nabla\nabla} A^\nabla$. By applying the last claim for $\bar{A}^{\nabla}$ we can conclude
$$\bar{A}^\nabla =\bar{A}^\nabla (\bar{A}^{\nabla\nabla} \bar{A}^\nabla) = \bar{A}^\nabla \bar{A}^{\nabla\nabla} = \bar{A}^{\nabla\nabla}.$$

b. According to Lemma 5.7 we have $A^{\nabla\nabla}=(P\bar{A})^{\nabla\nabla}=P\bar{A}^{\nabla\nabla}= P\bar{A}^{\nabla}=PA^\nabla P$.
\end{proof}

\vskip 0.25 truecm

\begin{df}
 The $\mathbf{(i_1,..., i_k)}$-\textbf{minor} $M_{i_1,...,i_k}$ of a matrix $A = (a_{i,j})$ is obtained by deleting the $i_1,...,i_k$ rows of $A$ and their corresponding columns.
\end{df}

\vskip 0.25 truecm
Let $\mathcal{E}$ denote the monoid of matrices that are factorizable over $\mathbb{T}$, and $\mathcal{E}_R$  the monoid of matrices that are factorizable over $R$ ($\mathcal{E}_R\subseteq \mathcal{E}$). We notice that unlike the classical case, $\mathcal{E}$ does not coincide with the group of invertible matrices, or with the non-singular matrices denoted by $\mathcal{R}$.

 For example 
$\left(
\begin{array}{ccc}
1_R &1_R & 0_R\\
1_R &  1_R&0_R\\
0_R &  0_R & 1_R
\end{array}
\right)\in\mathcal{E}\setminus \mathcal{R}\ $ and
$\ \left(
\begin{array}{ccc}
1_R &1_R & 0_R\\
0_R &  1_R&1_R\\
2_R &  0_R & 1_R
\end{array}
\right)\in \mathcal{R}\setminus\mathcal{E}.$

\vskip 0.5 truecm

\begin{pro}
Let $A$ be an $n\times n$ matrix in normal form. If an $n-1\times n-1$ minor~$M_{i_n}$ of $A$ is in~$\mathcal{E}$, then $A\in\mathcal{E}$  if \begin{equation}a_{i_n,j}\geq a_{i_n,k}a_{k,j}\ \forall i_n\ne j \ and\ \forall k\ne i_n,j.\end{equation} $A\in\mathcal{E}_R$ if the inequality in (6.7) is strict, or if $a_{i_n,j}$ is ghost in case of equality.
\end{pro}

\begin{proof}
The ${i_n}^{th}$ column may be obtained by applying $E_{i_t+(a_{i_t,i_n})\cdot {i_n}^{th}row},\ \forall t\ne n$. Then, by applying $E_{i_n+(a_{i_n,k})\cdot k^{th}row}\ \forall k\ne i_n$, we obtain $$a_{i_n,j}+\sum_{k\ne i_n,j} a_{i_n,k}a_{k,j}$$ in the $i_n,j$ position, for every $i_n\ne j$, which is $a_{i_n,j}$ if $a_{i_n,j}\geq a_{i_n,k}a_{k,j}\ \forall k$.
\vskip 0.7 truecm

$$\left(
\begin{array}{cccc}
n-1\times n-1 &minor & | & *\\
has\ \ been & obtained &| &\vdots\\
--- &  --- & 0 & *\\
&&&\\
\sum a_{n,k}a_{k,1}& \cdots &\sum a_{n,k}a_{k,n-1} & 0
\end{array}
\right)$$

\begin{center}\tiny{Figure 2. It is easy to obtain the $n^{th}$  column, using the 0 in the $n^{th}$  row. We obtain the $n^{th}$  row, using the $0$'s on the diagonal:  $a_{n,j}+\sum_{k\ne j}a_{n,k}a_{k,j}$}\end{center}

\end{proof}

\vskip 0.25 truecm

\begin{lem}
If $A$ is in normal form, then any $m\times m$-minor $M_{i_{m+1},...,i_n}$ of  $A^{\nabla}$ is in~$\mathcal{E}$, and  any $m\times m$-minor $M'_{i_{m+1},...,i_n}$ of  $A^{\nabla\nabla}$ is in~$\mathcal{E_R}$.
\end{lem}

\vskip 0.5 truecm

\begin{proof}
We can obtain any $2\times 2$ minor $M_{i_3,...,i_{n}}$ by embedding the matrices in the algorithm of Example 2.9-part(b) into rows and columns $i_1$ and $i_2$  in $A^{\nabla}$. Inductively, we assume that  any  $m-1\times m-1$-minor $M_{i_{m},...,i_n}$ of  $A^{\nabla}$   is in $\mathcal{E}$, and show that~$M_{i_{m+1},...,i_n}\in\mathcal{E}$.

By Proposition 6.4, in order to recover row $i_m$, we need to verify that  $$a'_{i_m,j}\geq a'_{i_m,k}a'_{k,j}\ \forall i_m\ne j\  \text{and}\ \forall k\ne j, i_m,...,i_n,$$ (for $\mathcal{E}_R$ we need strictly bigger, or equal and ghost) where $a'_{r,s}$ are the entries of $A^{\nabla}$.  

By  Corollary 6.2 the entries of~$A^{\nabla\nabla}$ are~$\nu$-equivalent to the entries of~$A^{\nabla}$.  If there exist~$m,j$ and $k$ such that $a'_{i_m,j}=a'_{i_m,k}a'_{k,j}$, and $a'_{i_m,j}\in T$,  then we cannot obtain $a'_{i_m,j}$ in $R$ by using $a_{i_m,j}+\sum_{k\ne j}a_{i_m,k}a_{k,j}$. However, the $i_m,j$ position in $A^{\nabla\nabla}$ in this case  is~$(a'_{i_m,j})^{\nu}$, since $a'_{i_m,k}a'_{k,j}$ itself is a summand in the $i_m,j$ position in $A^{\nabla\nabla}$ (corresponding to the permutation $(i_m\ k\ j)\cdot Id$), creating a second leading summand.  Hence, using Proposition 6.4, we may use the entries of $A^{\nabla}$, and the conclusion  applies to $A^{\nabla\nabla}$ in~$\mathcal{E}_R$ as well. 

The indices of our minor are $i_1,...,i_m$ and we need to show that

\begin{equation} \label{geq}a'_{i_m,j}=\sum_{\tiny{\begin{array}{cc}\pi\in S_m:\\\pi(j)=i_m\end{array}}}(a_{i_m,\pi(i_m)}\cdots a_{\pi^{-1}(j),j})C_{\pi}\end{equation} dominates the product between  \begin{equation}\label{leq}a'_{i_m,k}=\sum_{\tiny{\begin{array}{cc}\sigma\in S_m:\\\sigma(k)=i_m\end{array}}}(a_{i_m,\sigma(i_m)}\cdots a_{\sigma^{-1}(k),k})C_{\sigma}\end{equation} and \begin{equation}\label{leq1}a'_{k,j}=\sum_{\tiny{\begin{array}{cc}\tau\in S_m:\\\tau(j)=k\end{array}}}(a_{k,\tau(k)}\cdots a_{\tau^{-1}(j),j})C_{\tau},\end{equation} for every $ i_m\ne j \ and\ k\ne j, i_m,...,i_n$, where $C_{\pi}, C_{\sigma}$ and $C_{\tau}$ are the products of the remaining  cycle tracks in~$\pi, \sigma$ and $\tau$ respectively. 

\vskip 0.25 truecm

Indeed, every summand in the product of \eqref{leq} and \eqref{leq1} can be factored into a product of a term from \eqref{geq} and permutation tracks of $A$ which are smaller than $1_R$, causing each summand in the product of \eqref{leq} and \eqref{leq1} to be dominated by some summand of \eqref{geq}, as desired. We prove this property, similarly to Claim 6.1, by looking at the  cycles of $i_m$ and $j$ in the terms of $\sigma$ and $\tau$, for every~$\sigma,\ \tau\in S_m$ such that~$\sigma(k)=i_m,\ \tau(j)=k$. For both permutations we need to distinguish between cases in which $i_m$ and $j$ are in the same  cycle or not. Hence, there are two possible types of summands:

\noindent\underline{Case I- $i_m$ and $j$ share a common  cycle in at least one of the permutations:}

\noindent (Without loss of generality, we may assume that they share a cycle in $\tau$)

\small$$\left[(a_{k,\tau(k)}\cdots a_{\tau^{-1}(i_m),i_m}\underline{a_{i_m,\tau(i_m)}\cdots a_{\tau^{-1}(j),j}})C_{\tau}\right]\left[ (a_{i_m,\sigma(i_m)}\cdots a_{\sigma^{-1}(k),k})C_{\sigma}\right]$$\normalsize  
 \normalsize The underlined sequence, when composed with disjoint Id cycle tracks, is  a summand of \eqref{geq}. The rest of the monomial 
$$(a_{k,\tau(k)}\cdots a_{\tau^{-1}(i_m),i_m})(a_{i_m,\sigma(i_m)}\cdots a_{\sigma^{-1}(k),k})C_{\tau}C_{\sigma}$$
is composed from cycle tracks (if there are repeating indices in $$(a_{k,\tau(k)}\cdots a_{\tau^{-1}(i_m),i_m})(a_{i_m,\sigma(i_m)}\cdots a_{\sigma^{-1}(k),k}),$$ each part itself  being factored into cycle tracks, starting and ending at the points of repetition),  can be viewed as a permutation track of $A$ when composed with the~Id, and therefore is dominated by~$1_R$.
\vskip 0.25 truecm

\noindent \underline{Case II- Where $i_m$ and $j$ do not share a common  cycle, neither in $\sigma$ or in~$\tau$:}
\small$$\left[(\underline{a_{i_m,\sigma(i_m)}\cdots a_{\sigma^{-1}(k),k}})(a_{j,\sigma(j)}\cdots a_{\sigma^{-1}(j),j})C'_{\sigma}\right]\left[(\underline{a_{k,\tau(k)}\cdots a_{\tau^{-1}(j),j}})(a_{i_m,\tau(i_m)}\cdots a_{\tau^{-1}(i_m),i_m})C'_{\tau}\right].$$\normalsize
We compose the two underlined sequences, which, when composed with disjoint Id cycle tracks, would \textit{include}  a summand from \eqref{geq} (once again, if there are repeating indices, then this part itself is being factored into cycle tracks, starting and ending at the points of repetition, and a summand from \eqref{geq}). The rest of the monomial is composed once again from  simple cycle tracks, and therefore is dominated by~$1_R$.

\end{proof}

\vskip 0.25 truecm

\begin{cor}
If $A$ is a not strictly singular matrix over $\mathbb{T}$ (respectively non-singular over~$R$), then $A^{\nabla}$ (respectively  $A^{\nabla\nabla}$) is factorizable. 
\end{cor}

\begin{proof}
Let $\bar{A}$ be the normal form of $A$. By Lemma 6.5 we may factor any $k\times k$ minor of~$\ \bar{A}^{\nabla}$, including  $\ \bar{A}^{\nabla}$ itself. Using Lemma 5.7 where $A=P\bar{A}$, we get that~$A^{\nabla}=\bar{A}^{\nabla}P^{-1}$ is factorizable over $\mathbb{T}$ as well.
One concludes  immediately that $A^{\nabla\nabla}=P\bar{A}^{\nabla\nabla}\cong_{\nu}P\bar{A}^{\nabla}$ is also factorizable over $\mathbb{T}$ and $R$.
\end{proof}

Noticing that $A^{\nabla}$ rises from supertropical algebraic considerations, we would like to make the connection to the familiar tropical concept of $A^*$. According to Remark 2.7, we can conclude from the LDM factorization of $A^*$ in ~\cite{LDM} that $A^*$ is factorizable.

\begin{lem} If~$A$ is a matrix of order~$n$ in normal form and~$k$ be a natural number such that~$k\geq n-1$, then~$A^{\nabla}=A^{k}$. In particular, $A^k=A^{k+1},\ \forall k\geq n-1$.

\noindent (Equalities are being interpreted as $\cong_{\nu}$ over $R$).
\end{lem}

\begin {proof}
Let $A$ be a matrix of order $n$ in normal form. The diagonal entries of $A^{\nabla}$ and~$A^{n-1}$ are $1_R$. We will show equality for the off-diagonal part.

As we saw before, the off-diagonal $i,j$ position in $A^{\nabla}$ is 
$$\sum [a_{i,\pi(i)}\cdots a_{\pi^{-1}(j),j}]C_{\pi}.$$ Of course each summand is a product of $n-1$ entries of A. According to Remark 4.1, this sum is being dominated by the summands 
$$\sum [a_{i,\pi(i)}\cdots a_{\pi^{-1}(j),j}]C_{Id},$$ where $C_{Id}$ is a product of $Id$ cycle tracks.
\vskip 0.25 truecm

The off-diagonal $i,j$ position in $A^{n-1}$  is 
$$\sum [a_{i,t_1}a_{t_1,t_2}\cdots a_{t_{n-2},j}]$$ (each summand is a product of $n-1$ entries of A). For every repeating index $t_k$ we can start and end a cycle track at the points of repetition, obtaining summands of the form
$$\sum(a_{i,t_1}\cdots a_{t_l-1,t_l})C_1\cdots C_v(a_{t_l,t_{l+1}}\cdots a_{t_{n-2},j})=
\sum (a_{i,t_1}\cdots a_{t_l-1,t_l}a_{t_l,t_{l+1}}\cdots a_{t_{n-2},j})C_1\cdots C_v,$$ where $C_1,...,C_v$ are the cycles obtained where the indices repeat. Also, there are no indices repeating in $\{i,t_1,\cdots ,t_l,\cdots ,t_{n-2},j\}$, which indicates a cycle track missing $a_{j,i}$. These summands are dominated by the summands 
$$\sum (a_{i,t_1}\cdots a_{t_l-1,t_l}a_{t_l,t_{l+1}}\cdots a_{t_{n-2},j})C_{Id}.$$
Therefore, for $A$ in normal form we have $A^{\nabla}=A^{n-1}$.
Next, we saw in Claim 6.1 that~$AA^{\nabla}=A^{\nabla}$ for $A$ in normal form. Therefore $A^n=AA^{n-1}=AA^{\nabla}=A^{\nabla}$. Inductively we get $ A^k=A^{\nabla},\ \forall k\geq n-1$.

\end{proof}

\begin{cla}
$A^{\nabla}=A^*$ when $A$ is in normal form .  (Equality is being interpreted as $\cong_{\nu}$ over $R$).

\end{cla}

\begin{proof}

Looking at the definition of $A^*$ appears in ~\cite{LDM} one can see that for $A$ in normal form \begin{equation}A^*=\sum_{i\in \mathbb{N}\bigcup{0}} A^i.\end{equation} Moreover, for $A$ in normal form $(A^{k+1})_{i,j}=(A^k)_{i,j}+B,\ \forall k\in \mathbb{N}\bigcup{0}$, for some matrix~$B$ of order $n$. That is, the~$i,j$ position in every power of the matrix $A$ is a summand in the~ $i,j$ position in the subsequent power of $A$, which means each position can only increase comparing to the corresponding position in the former power. 
Therefore, $$A^*=A^{n-1}=A^{\nabla},$$ for $A$ of order $n$ in normal form.

\end{proof}

In fact, the last Claim yields an alternate proof of the factorizability of $A^\nabla$ from the LDM factorization of $A^*$ obtained independently in ~\cite{LDM}.

\end{document}